\documentclass[]{ensmath}

\EnsMath pages 1-20 

\usepackage{epsfig}
\usepackage{xcolor}

\usepackage[all,cmtip]{xy}

\title[Kan Complexes, Homotopy and Cohomology]

\author[J. STEINEBRUNNER]{Jan \sn{Steinebrunner}}

\newcommand{\N}{\mathbb{N}}
\newcommand{\Z}{\mathbb{Z}}
\newcommand{\R}{\mathbb{R}}
\newcommand{\Top}{{\mathcal{T}op}}
\newcommand{\CW}{\mathcal{CW}}
\newcommand{\Set}{{\mathcal{S}et}}
\newcommand{\sSet}{{\textbf{s}\mathcal{S}et}}
\newcommand{\sSk}{{\textbf{s}\mathcal{S}k}}
\newcommand{\sPair}{{\textbf{s}\mathcal{P}air}}

\newcommand{\Ab}{{\mathcal{A}b}}
\newcommand{\Kan}{{\mathfrak{Kan}}}

\newcommand{\vD}{\varDelta}
\newcommand{\p}{\partial}
\newcommand{\pt}{\cdot}
\newcommand{\inj}{\hookrightarrow}

\newcommand{\op}[1]{\operatorname{#1}}
\newcommand{\image}{\op{im}}

\newcommand{\bel}{\;?\;}
\newcommand{\s}{\star}
\newcommand{\de}{\mathbf{d}}

\newcommand{\ga}{\alpha}              
\newcommand{\gb}{\beta}

\newcommand{\gi}{\iota}

\newcommand{\gp}{\varphi}

\newcommand{\gO}{\Omega}

\begin{document}

\begin{center}
\large
 \textsc{Kan Complexes, Homotopy and Cohomology}
 
\normalsize
 by \textsc{Jan Steinebrunner}
 \bigskip
\end{center}

\begin{abstract}
	This article shows several new methods for proofs on Kan complexes
	while using them to give a compact introduction to the homotopy groups of these complexes.
	Then more advanced objects are studied starting with homology and the Hurewicz homomorphism. 
	Eilenberg-Mac Lane-spaces are constructed explicitly and can then be used to define spectral cohomology. 
	In the end the equivalence to the usual simplicial cohomology is shown.
\end{abstract}
\tableofcontents


\section*{Motivation}\label{sec:motiv}
	Kan complexes are combinatorial objects similar to simplicial complexes with an additional 
	property making it possible to define a notion of homotopy, homotopy groups and many other common
	constructions in topology. 
	The theory of Kan complexes requires almost no previous knowledge.
	Yet, most of the literature on this topic is aimed at advanced mathematicians, 
	who are already familiar with algebraic topology.
	
	This article will make the theory accessible to non-experts and undergraduate students. 
	The main obstacle, however, is that many of the fundamental proofs are indeed elementary, 
	though often, due to an involved combinatorial structure, quite confusing. 
	Thus we will give several methods for simplifying and structuring such arguments.
	In the first part they will be used to give a compact and basic introduction into combinatorial homotopy groups.
	The second part deals with (co)homology and applies the techniques to
	give elementary proofs for known statements such as Hurewicz's theorem
	or the equivalence of spectral and simplicial cohomology in the category of simplicial sets.\\
	
	After giving all basic definitions in section \ref{sec:basics}, 
	we define cycles and fillings. These geometrical concepts help to make arguments more intuitive.
	In section \ref{sec:topology} we will get a short insight in how to 'identify' Kan complexes with topological spaces.
		\footnote{This requires the reader to know some basics of topology, but it is only meant to improve geometrical intuition 
		and to show applications. Thus, it is helpful but not necessary for the rest of this article.}
	Then a definition of homotopy can be given, fitting the intuition gained by this identification. 
	In the Key lemma we can prove an equivalent characterization of homotopy, 
	which is more useful to work with.
	This makes it easy to introduce homotopy groups and the exact sequence of a pair
	quickly without mentioning much combinatorics.

	Part two starts in section \ref{sec:homology} by introducing simplicial (co)homology. 
	This then leads to an elementary proof of Hurewicz's theorem in a special case.
	Using a 'trivial completion' we will be able to construct Eilenberg-Mac Lane spaces explicitly.
	These spaces then have to be isomorphic to those obtained via the Dold-Kan correspondence.
	Simplical maps to them define the spectral cohomology, by this we mean the cohomology theory represented by a spectrum.
	We get an elementary identification of the cocycles and coboudaries in the spectral setting
	with those of simplicial one. This yields a natural isomorphism of the cohomology theories on
	the simplicial sets. 
	A statement, which we do just mention, but not prove here, can then be used to transfer these results to the topological category.

\section{Basic constructions}\label{sec:basics}
	
	We shortly set up the basic definitions. A more detailed introduction can, for example, be found in \cite{May} or \cite{GJ}.
	Simplicial sets can be understood as a generalization of simplicial complexes. 
	There is a set of '$n$-simplices' for each dimension $n\in \N_0$ 
	and so-called face operators are used to keep track of which faces each simplex has.
	The best way to understand the degeneracy operators is to regard
	them in the special case of example \ref{def:stsimplex}.
	
\begin{definition}[Simplicial set]\label{def:simset}
	A simplicial set $K_\pt$ is a sequence of disjoint sets $K_n$ for $n\in \N_0$ together with maps 
	\[
		d_i = d_{i,n}\: K_n \to K_{n-1} \; \text{ as well as } \; s_i = s_{i,n} \: K_n \to K_{n+1}, 
	\] 
	called face and degeneracy operators for $i\le n \in \N_0$, such that for all $j\le n\in \N_0$:
	\begin{align}
		d_i d_j = d_{j-1} d_i				\quad& \text{ if } 	i<j	\\
		s_i s_j = s_{j+1} s_i 				\quad& \text{ if } 	i\le j	\\
		d_i s_j = s_{j-1} d_i 				\quad& \text{ if } 	i<j	\\
		d_i s_i = id_{K_n} = d_{i+1} s_i 		&				\\
		d_i s_j = s_j d_{i-1} 				\quad& \text{ if } 	i>j+1.
	\end{align}
\end{definition}

\begin{definition}
	We denote the $(n+1)$-tuple of faces of $x$ by $\p x := (d_0x,\dots,d_nx)$.
	A subcomplex $L_\pt$ of a simplicial set $K_\pt$ is a sequence of subsets $L_n \subset K_n$
	forming a simplicial set with the restricted face and degeneracy operators.
	A simplicial Pair $(K_\pt,L_\pt)$ is a simplicial set $K_\pt$ with a distinguished subcomplex $L_\pt \subset K_\pt$.
	A pointed simplicial set $(K_\pt,\s)$ is a simplicial set with a distinguished base point $\s \in K_0$.
	A pointed simplicial $(K_\pt,L_\pt,\s)$ Pair has its base point in the distinguished subcomplex: $\s \in L_0$.
	
	A simplex of the form $y=s_i x$ is called degenerate.
\end{definition}
	
\begin{definition}
	A map $f$ between simplicial sets $K_\pt,L_\pt$ is a sequence of maps $f_n\:K_n\to L_n$ compatible
	with face and degeneracy operators: $f_nd_i=d_if_n$ and $ f_ns_i=s_if_n.$ Maps of pointed simplicial sets
	or simplicial Pairs $f\:(K_\pt,L_\pt)\to (K_\pt',L_\pt')$ should preserve this extra structure, 
	so $f(\s)=\s$ and $f(L_\pt) \subset L_\pt'$.
	
	We get categories of simplical sets $\sSet$, pointed simplicial sets $\sSet_+$,
	simplicial pairs $\sPair$ and pointed simplicial pairs $\sPair_+$.
\end{definition}
	
\begin{definition}\label{def:stsimplex}
	Let $[n]$ denote the totally ordered set $\left(\{0,\dots,n\},\ge\right)$ for any $n\in \N_0$. 
	Then the standard $n$-simplex $\vD_\pt^n$ is the simplicial set consisting of order preserving maps:
	\[
		\vD_k^n := \{x:[k]\to[n]\;|\;x\text{ non-decreasing}\} = \{ (x_0,\dots,x_k) \in [n]^{k+1} \;|\; x_0 \le \dots \le x_k\}.
	\]
	With the face and degeneracy operators given by omitting and repeating entries:
	\[
		d_i(x_0,\dots,x_n) = (x_0,\dots,x_{i-1},x_{i+1},\dots) \text{ and } s_i(x_0,\dots,x_n) = (x_0,\dots,x_{i},x_{i},\dots).
	\]
	This simplicial set has two interesting subcomplexes: $\p \vD_\pt$ contains all simplices $x:[k]\to[n]$,
	such that $X$ is not onto. The $(n,k)$-horn $\Lambda_\pt^n$ contains all simplices,
	for which $\operatorname{Im}(x)=x([k])$ does not contain the set $[n]\setminus \{k\} = \{0,\dots,k-1,k+1,\dots,n\}$.
\end{definition}
	
\begin{remarque} 
	$\vD^n$ is the generic example of a simplicial set, the rules for the operators are clearly satisfied. 
	One can think of $\vD^n$ as the convex hull of $n+1$ affine-linearly independent points $p_0, \dots, p_n\in \R^n$.
	Then a simplex $x=(x_0,\dots,x_k)$ is the convex hull of the points $p_{x_0},\dots,p_{x_k}$. 
	In this picture $d_i$ and $s_i$ really do what one expects.
	Then $\p\vD_\pt^n$ is the boundary of $\vD_\pt^n$ and $\Lambda_k^n$ is $\p\vD_\pt^n$ without the $k$th face.
	Here a degenerate simplex is a simplex for which its 'geometrical' dimension is lower than 'combinatorial' dimension,
	these are exactly those simplices in which an entry occurs twice.
\end{remarque}
	
\begin{remarque} 
	A base point $\s \in K_\pt$ generates a subcomplex $S_\pt \subset K_\pt$. It is easily seen that this subcomplex 
	has only one simplex in each dimension, which we will denote by $\s_n$ or just $\s$, we even call $S_\pt$ only $\s$
	whenever this does not lead to confusion
\end{remarque}

\begin{definition}
	For two simplical sets $K_\pt$ and $L_\pt$ their product is defined by $(K\times L)_n:=K_n \times L_n$,
	with the face and degeneracy operators: 
	\[
		d_i(x,y) = (d_i x, d_i y) \text{ and } s_i(x,y) = (s_i x, s_i y) \;\;\text{ for } (x,y)\in K_n\times L_n \text{ and }i\in [n].
	\]

	The coproduct or disjoint union is given by $(K \amalg L)_n := K_n \amalg L_n$
	with operators such that $K_\pt,L_\pt$ are subcomplexes.
	Arbitrary (co-)products are defined in a similar manner.
\end{definition}

\section{The Kan Property}\label{sec:kan}

	We introduce a notion of cycles and boundaries. We will see, for example in lemma \ref{lem:matrix}, 
	that their properties fit the topological intuition.

\begin{definition}
	An $(n+2)$ tuple $(x_0,\dots,x_{n+1}) \in (K_{n})^{n+2}$ of $n$-simplices in a simplicial set 
	is called $n$-\emph{cycle} if it is \emph{compatible} in the sense that it satisfies $d_i x_j = d_{j-1} x_i$ for all $i<j\in [n+1]$.
	It is an $n$-\emph{boundary} if there is an $(n+1)$-simplex $x$ such that $\p x = (x_0,\dots,x_{n+1})$.
		\footnote{Cycles, boundaries and horns are all enumerated by the dimension of the simplices, which they contain. 
			Thus an $n$-boundary is filled by an $(n+1)$ simplex.}
	
	An $(n+2)$ tuple $(x_0,\dots,x_{k-1},\bel,x_{k+1},\dots,x_{n+1})$, where the $k$th simplex is not 
	specified yet, is called $(n,k)$-horn, if it is compatible in the above sense.
	An $n$-simplex $x_k$ completing the horn to a boundary $(x_0,\dots,x_{n+1})$ is called \emph{completion}
	of the horn. An $(n+1)$-simplex filling the corresponding boundary is called a \emph{filling} of the horn.
	%
\end{definition}

\begin{remarque}
	Every boundary is a cycle, because $d_id_jx=d_{j-1}d_ix$. There is a one-to-one correspondence 
	between the $n$ simplices of $K_n$ of $K_\pt$ and the simplicial maps $\vD^n\to K_\pt$. 
	Similarly the $n$-cycles are one-to-one to the simplicial maps $\p\vD^{n+1} \to K_\pt$
	and the $(n,k)$-horns correspond to the maps $\Lambda_k^{n+1}\to K_\pt$.
\end{remarque}

\begin{definition}
	A simplicial set $K_\pt$ is called Kan complex if every horn has a filling.
\end{definition}
	
	One can think of the above definitions as indicated in the following picture. 
	For the moment this should just provide some intuition,
	but using the definition from section \ref{sec:topology} this becomes meaningful,
	as it can be interpreted as simplices in $S_\pt(\R^2\setminus \{\text{holes}\})$.
	
	\begfig{0.0cm}
	\center \includegraphics[width=0.8\textwidth]{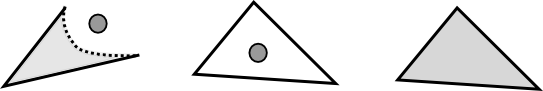}
	\Figure1{A $(1,2)$-horn with dashed completion and grey filling, a $1$-cycle that cannot be filled 
		 and a $1$ boundary with filling in $S_\pt(\R^2\setminus \{\text{holes}\})$}
	\endfig

	The next two ideas follow from the Kan property and are very important for this article, 
	we will see a lot of applications in the sections \ref{sec:homotopy} to \ref{sec:exact}.
	
\begin{definition}\label{def:filler}
	For an $(n,k)$-horn $(x_0,\dots,\bel,\dots,x_{n+1})$ in a Kan complex we write 
	\[
		(x_0,\dots,x_{k-1},\underline{y},x_{k+1},\dots,x_{n+1})
	\]
	and mean that $y$ is defined to be some simplex that completes the horn.
\end{definition}
	
\begin{definition}
	An $(n+2)$ tuple of $(n-1)$-cycles $(c^0,\dots,c^{n+1}),\; c^i=(x_0^i,\dots,x_n^i)$ is called compatible 
	if it satisfies $x_j^i = x_i^{j-1}$ for all $i<j\in [n+1]$. That is that the matrix $((c^0)^T,\dots,(c^{n+1})^T)$
	has the following shifted mirror symmetry:
	\[
		\left(\color{gray}
		\begin{matrix} 
			x_0^0 	\\ x_1^0 \\   \vdots \\  x_{n}^0	\\ x_{n+1}^0
		\end{matrix}\quad
		\begin{matrix}
			x_0^0	\\ \color{black} x_1^1 \\  \color{black} \vdots \\  \color{black}x_{n}^1\\ \color{black}x_{n+1}^1
		\end{matrix}\quad
		\begin{matrix}
			x_1^0	\\ \color{black}  x_1^1 \\   \vdots \\  x_{n}^2	\\ x_{n+1}^2
		\end{matrix}\quad
		\dots
		\begin{matrix}
			x_{n-1}^0\\ \color{black}  x_{n-1}^1    \\   \vdots \\  \color{black}x_{n}^{n}	\\\color{black} x_{n+1}^{n}
		\end{matrix}\quad
		\begin{matrix}
			x_n^0\\ \color{black}  x_n^1 \\   \vdots \\  \color{black}x_{n}^{n}	\\ x_{n+1}^{n+1}
		\end{matrix}\quad
		\begin{matrix}
			x_{n+1}^0\\ \color{black}  x_{n+1}^1 \\   \vdots \\  \color{black}x_{n+1}^{n}	\\ x_{n+1}^{n+1}
		\end{matrix}
		\color{black}
		\right).
	\]
\end{definition}

\begin{lemma}\label{lem:matrix}
	If $(c^0,\dots,c^{n+1})$ are compatible cycles in a Kan complex and all $c_i$ but $c_k$ are known to be boundaries, 
	then $c_k$ is a boundary as well.
\end{lemma}
\begin{proof}
	Let $\p y_i = c^i$ for $n$-simplices $y_i$ and $i\ne k$. Then the compatibility of  $(y_0,\dots,y_{k-1},\bel,y_{k+1},\dots,y_{n+1})$ 
	is implied by the compatibility of $c^i$, since $d_j y_i = x_j^i$.
	Using the Kan property we can complete this horn by some $y_k$. 
	Then $\p y_k$ has to be $c^k$, as it is, due to the compatibility condition, completely determined by the other $y_i$'s: 
	\[
		d_i y_k = d_{k-1} y_i = c_i^{k-1} = c_k^i \text{ for }i< k \text{ and } d_i y_k = d_k y_{i+1} = c_{i+1}^k = c_k^i \text{ for } k\le i.
	\]
	In other words: One column in a matrix with the symmetry is completely determined by the others.
	Either way, $y_k$ shows that $c^k$ is a boundary.
\end{proof}
\begin{example}
	As an example we regard this in the case $n=1$. A $0$-cycle $c=(x,y)$ consists of two points $x,y\in K_0$.
	It can be filled if there is a line $h\in K_1$ that has $x$ and $y$ as its endpoints:
	$d_0h=x$ and $d_1h=y$, or equivalently $\p h=(x,y)$. A compatible set of such cycles would be
	$((x,y),(x,z),(y,z))$. Then the above statement is for $k=0$:
	If there is a line from $x$ to $z$ and from $y$ to $z$, then there is also a line from $x$ to $y$.
\end{example}

\begin{lemma}
	Arbitrary (co-)products of Kan complexes are again Kan complexes.
\end{lemma}
\begin{proof}
	We restrict our attention to the case of a product of two complexes.
	The case of an arbitrary index set is very similar and the coproduct case is also not too difficult.
	
	Let $(z_0,\dots,\bel,\dots,z_n)$ be a horn in $(K\times L)_\pt$. 
	We can write $z_i = (x_i,y_i)$ and due to the definition of the face operators in a product,
	$(x_0,\dots,\bel,\dots,x_n)$ and $(y_0,\dots,\bel,\dots,y_n)$ are compatible in $K_\pt$ resp. $L_\pt$.
	Using the Kan properties we can complete both horns separately by $x_k$ and $y_k$,
	but then $z_k:=(x_k,y_k)$ completes the first horn.
\end{proof}

\begin{remark}\label{rem:noquo}
	A quotient of two Kan complexes does not have to be a Kan complex again.
	As a general rule, an operation of simplicial sets can only preserve the Kan property,
	if its action on the homotopy groups can be easily understood.

	This is also, why we are not able to give interesting examples of Kan complexes at the moment.
	For instance $\vD^n$, $\p\vD^n$ and $\Lambda_k^{n}$ do not satisfy the Kan property,
	whereas the trivial complex $\s$ does.
	A very important example will be the singular complex $S_\pt(X)$ of a topological space $X$ in section \ref{sec:topology}. 
	The most interesting examples will be the Eilenberg-Mac Lane spaces $H(A,n)_\pt$
	for each abelian group $A\in \Ab$ and natural number $n\in \N_{>0}$ in chapter \ref{sec:EML}. 
	They are spaces with only one non-trivial homotopy group $\pi_n(H(A,n)_\pt)=A$.
\end{remark}

\section{Connections to Topology}\label{sec:topology}
We now define two Functors between $\sSet$ and $\Top$, which will make it possible
to translate statements about one category into the other one. 

\begin{definition}
	The topological standard $n$-simplex $|\vD^n|$ is for $n\in \N_0$ defined by 
	\[
		|\vD^n| := \{(v^0,\dots,v^n) \in \R^{n+1} \; | \; \forall i \in [n]\: v_i\ge 0 \text{ and } \sum\nolimits_{i=0}^{n}v_i = 1\}.
	\]
	We get maps $\delta_i\:|\vD^{n-1}|\to |\vD^{n}|$ and $\sigma_i\:|\vD^{n+1}|\to|\vD^{n}|$ for $i\in [n]$ by
	\begin{align*}
		&\delta_i(v^0,\dots,v^{n-1}) := (v^0,\dots,v^{i-1},0,v^{i},\dots,v^{n+1}) \\
		\text{ and } &\sigma_i(v^0,\dots,v^{n+1}) := (v^0,\dots,v^{i}+v^{i+1},\dots,v^{n+1}).
	\end{align*}
	For a simplicial set $K_\pt$ its (geometric) realization $|K_\pt|$ is the topological space given by
	\[
		\left(\prod\nolimits_{n\in\N_0} K_n \times |\vD^n| \right)\big/ \sim
	\]
	where the equivalence relation $\sim$ is generated by $(d_ix,v)\sim(x,\delta_iv)$ and 
	$(s_ix,w)\sim(x,\sigma_iw)$ for all $n\in\N, x\in K_n, v\in |\vD^{n-1}|, w\in |\vD^{n+1}|, i\in [n]$.
\end{definition}

\begin{remarque}
	The geometric realization is functor to $\CW$, the category of $\CW$-complexes. 
	The realization of the standard $n$-simplex is the topological
	standard $n$-simplex, so $|\vD^n|=|\vD^n|$ is still well-defined. The realization of $\p\vD_\pt^n$
	is homeomorphic to a sphere. May discusses this topic in part 2 of his book \cite{May}.
\end{remarque}

\begin{definition}
	For a topological space $X$ its singular complex $S_\pt(X)$ is given by
	\[
		S_n(X) := \Hom_\Top \left(|\vD^n|, X\right) 
	\]
	with the simplicial set structure $d_i f = f\circ \delta_i$ and $s_i f = f \circ \sigma_i$ for all $n\in\N, f\in S_n(X), i\in[n]$.
\end{definition}

\begin{remarque}\label{rem:singcomp}
	The singular complex $S_\pt(X)$ is a Kan complex for any space $X$, as one can easily see. 
	Just as with $\CW\subset\Top$, it seems convenient to restrict our attention to the Kan complexes $\Kan\subset\sSet$.
	The singular complex functor is a right adjoint to the geometric realization, i.e. there is 
	a natural isomorphism:
	\[
		\Hom_\Top(|K_\pt|,X) \cong \Hom_\sSet(K_\pt,S_\pt(X)).
	\]
	This statement can be found in chapter 15 of \cite{May}.
	For a good introduction into categories, functors, adjointness and so on,
	I recommend to read chapter 1 of Switzer's textbook \cite{Sw}.
\end{remarque}

\begin{theorem}\label{thm:|S|}
	For every Kan complex $K_\pt$ and every $\CW$ complex $X$, there are natural homotopy equivalences:
	\[
		K_\pt \sim S_\pt(|K_\pt|)	\quad \text{ and } \quad 	X \sim |S_\pt(X)|.
	\]
\end{theorem}
\begin{proof}
	The proof of this theorem involves a lot of work. A full exposition can, for example, be found 
	in chapter 16 of May's book \cite{May} or, with a more modern treatment, in I.11 of \cite{GJ}.
	We just give a short explanation of the statement:
	
	A homotopy equivalence is a map $f:X \to Y$ having a homotopy-inverse $g:Y \to X$, so $f\circ g \sim \id_Y$ 
	and $g\circ f \sim \id_X$. We have not defined homotopies for Kan complexes, but we will mention them in remark \ref{rem:homotopies}.
	In fact, homotopy equivalent spaces are very similar, so the functors $S_\pt$ and $|.|$ are, informally speaking,
	inverse up to homotopy. More formally, one can say that they induce an equivalence of the homotopy categories.
	
	We will not need this statement for any proof or theorem staying in the 'simplicial world'.
	However, it will be used very often, when we want to apply our results about Kan complexes to topological spaces.
	The most important examples for this will be the proof of Hurewicz's theorem for topological spaces in \ref{thm:hur}
	and the proof that the singular and the spectral cohomology theories are equivalent on $\Top$ in chapter \ref{sec:cohom}.
\end{proof}

\section{Homotopy}\label{sec:homotopy}

\begin{definition}
	Two $n$-simplices $x,y \in K_n$ in a Kan complex with the same boundary $\p x=\p y$ are said to be
	homotopic, written as $x\sim y$, if the following $n$-cycle is filled:
	\[
		(s_{n-1}d_0x, \dots, s_{n-1}d_{n-1}x, x, y).
	\]
	The filling $(n+1)$-simplex is called a homotopy, it is denoted by $h\: x\sim y$.
\end{definition}

\begin{remarque}
	One has to check that the given tuple is really a cycle, but this follows immediately from the 
	rules for $d_i$ and $s_i$ and the fact that $d_ix=d_iy$ for all $i\in [n]$.
	
	Geometrically speaking the definition states that $x\sim y$,
	if the space between $x$ and $y$ can be filled. Here the other components of $\p h$ are not important,
	as they are degenerate.
	
	The following Key-lemma will do the main part in proving that homotopy is an equivalence relation. 
	In addition it will justify our definition \ref{def:filler}, where then the new simplex is
	determined up to homotopy. 
	It states that replacing a simplex in a boundary by another one still gives a boundary, 
	if and only if the two simplices were homotopic.
	Thinking of $K_\pt = S_\pt(X)$ for some topological space $X$, this should be clear, 
	because we just have to glue the homotopy and the filling together, as the following picture illustrates:
\end{remarque}

	\begfig{0.0cm}
	\center \includegraphics[width=0.8\textwidth]{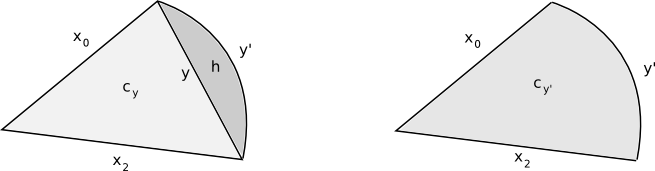}
	\Figure2{The Key lemma: A boundary and a homotopy can be glued to a new boundary.}
	\endfig

\begin{lemma}[Key-Lemma] \label{lem:homotopy}
	Let $c=(x_0,\dots,x_{k-1},y,x_{k+1},\dots,x_n)$ be an $(n-1)$-boundary in a Kan complex $K_\pt$, 
	and let $y'\in K_{n-1}$ satisfy $\p y' = \p y$, such that $c'=(x_0,\dots,y',\dots)$ is a cycle.
	Then $c'$ is a boundary if and only if $y \sim y'$. 
\end{lemma}
\begin{proof}
	We know that $c'$ is at least a cycle because we have $\p y=\p y'$. Now assume that it is a boundary.
	In order to show $y\sim y'$, we have to fill the cycle 
	\[
		h:=(s_{n-2}d_0y, \dots, s_{n-2}d_{n-2}y, y, y').
	\]
	We want to use lemma \ref{lem:matrix} for this.
	\\
(\textbf{Case:} $k\ne n$) We claim that 
	\[
		(\p (s_{n-1}x_0), \dots, h, \dots, \p (s_{n-1}x_{n-1}), c, c') \in (K_{n-1}^{n+1})^{n+2}
	\]
	is a compatible $(n+2)$ tuple of $(n-1)$-cycles. One can check this by regarding it in matrix form:
	\[ 
			\left(
			\begin{matrix}
			   d_0s_{n-1}x_0 \\ d_1s_{n-1}x_0 	\\   \vdots 	\\  d_{n-2}s_{n-1}x_0 	\\ 	x_0 	\\ x_0
			\end{matrix}\quad
			\begin{matrix}
			   d_0s_{n-1}x_1 \\ d_1s_{n-1}x_1 	\\   \vdots 	\\  d_{n-2}s_{n-1}x_1 	\\ 	x_1 	\\ x_1
			\end{matrix}\;
			\dots\,
		\color{blue}
			\begin{matrix}
			    s_{n-2}d_0y	 \\ s_{n-2}d_1y		\\   \vdots 	\\ s_{n-2}d_{n-2}y		\\  y    	\\ y'
			\end{matrix}\,
		\color{black}
			\dots\;
			\begin{matrix}
			   d_0s_{n-1}x_{n-1} \\ d_1s_{n-1}x_{n-1}\\   \vdots 	\\  d_{n-2}s_{n-1}x_{n-1}\\ x_{n-1}	\\ x_{n-1}
			\end{matrix}\quad
			\begin{matrix}
			    x_0 		 \\    	\vdots		\\   y 			\\    \vdots 			\\  x_{n-1} \\ x_n
			\end{matrix}\quad
		\color{blue}
			\begin{matrix}
			    x_0 		 \\    	\vdots		\\   y'			\\    \vdots 			\\  x_{n-1} \\ x_n
			\end{matrix}
		\color{black}
			\right).
	\]
	Now compatibility should follow by $d_i y= d_{k-1} x_i$ for $k<i$ and $d_i y = d_k x_{i+1}$ for $k\ge i$ as well as
	the rules of interaction between $s_i$ and $d_j$. 
	For example take $i<j\in[n-1]:$
	\[
		d_is_{n-1}x_j = s_{n-2}d_ix_j = s_{n-2}d_{j-1}x_i= d_{j-1}s_{n-1}x_i.
	\]
	
	Given the compatibility of the above matrix, lemma \ref{lem:matrix} tells us that $h$ is filled, since 
	$\p (s_{n-2}x_{i})$ has an obvious filling and $c$ and $c'$ are supposed to be boundaries. Thus $y$ is homotopic 
	to $y'$. Moreover the matrix shows that, given $h\:y\sim y'$ and therefore $h$ filled, $c'$ is a 
	boundary. So it proves both directions at once.\\
	\\
(\textbf{Case:} $k=n$) Here the $(n+2)$ tuple we used in the first case does not work, because $h$ would 
	have to be at the same position as $c_y$. But using 
	\[
		(\p (s_{n-1}x_0), \dots, \p (s_{n-1}x_{n-2}), c_y, c_{y'},h) \in (K_{n-1}^{n+1})^{n+2}
	\]
	yields a proof similar to the one above.
\end{proof}

\begin{corollary}\label{cor:homequi}
	Homotopy is an equivalence relation.
\end{corollary}
\begin{proof}
	Reflexivity follows by regarding the boundary of a degeneration of $x$:
	\[
		\p (s_n x) = (d_0 s_n x, \dots, d_{n+1} s_n x) = (s_{n-1} d_0 x, \dots, s_{n-1} d_{n-1} x, x, x).
	\]
	To show anti-transitivity let $h\:x\sim y$ and $h'\:x\sim z$ be homotopies of $n$-simplices.
	The Key-Lemma \ref{lem:homotopy} tells us that, if we replace $x$ in $\p h$ by $y$:
	\[
		\p h' = (s_{n-1} d_0 x, \dots, s_{n-1} d_{n-1} x, x, z) \;\;\to\;\; c = (s_{n-1} d_0 x, \dots, s_{n-1} d_{n-1} x, y, z)
	\] 
	the resulting cycle $c$ is still filled, because $y$ is homotopic to $z$. But then $c$ already is
	a homotopy from $y$ to $z$.
\end{proof}

\begin{definition}\label{def:relhom}
	Given a pair of Kan complexes $(K_\pt,L_\pt)$ we call two $n$-simplices $x,y\in K_n$ 
	homotopic relative $L_\pt$, if $d_i x = d_i y$ for $1\le i \le n$ and $d_0 x, d_0 y \in L_{n-1}$,
	such that for some $h_L\in L_n$ the cycle 
	\[
		c=(h_L, s_{n-1} d_1 x, \dots, s_{n-1} d_{n-1} x, x, y)
	\]	
	is filled by some $h\in K_{n+1}$. This $h$ is called relative homotopy, denoted by
	$h:x \sim_{\op{rel}L} y$.
\end{definition}

\begin{remarque}\label{rem:relcase}
	The relative case is often very similar to the 'absolute' one. Here one should also check
	that $c$ is a cycle for an appropriate $h_L$. 
	
	Now, apply $\p$ to each entry of this cycle. Then,
	by regarding the obtained matrix, one sees that $h_L$ has to be a homotopy $d_0 x \sim d_0 y$ in $L_\pt$.
\end{remarque}

\begin{lemma}[relative Key-Lemma]\label{lem:relhom}
	For a pair of Kan complexes $(K_\pt,L_\pt)$ we have:
	Let $b = (l,x_1, \dots, x_n)$ be an $(n-1)$-boundary in $K_\pt$ where $l\in L_{n-1}$
	and let $b' := (l', x_1, \dots, x_k', \dots, x_n)$. Then there is an $l'\in L_{n-1}$ such that $b'$
	is filled, if and only if $x_k$ is homotopic to $x_k'$ relative $L_\pt$.
\end{lemma}
\begin{proof}
	The argumentation is very similar to the one in lemma \ref{lem:homotopy}, but here one has to do 
	some extra work for the $l$ part. We only work in the case $k\ne n+1$, the remaining one then follows quite similarly.
	
	For the first direction of the proof let $h\:x_k \sim_{\op{rel}L} x_k'$ be a relative homotopy.
	The following $(n-1,k)$-horn lies completely in $L_\pt$, because $d_0 h \in L_\pt$ by definition 
	and $s_{n-2} d_0 x_i = s_{n-2} d_{i-1} l \in L_\pt$. Thus we can use the Kan property of $L_\pt$ to define $l'\in L_{n-1}$ to complete:
	\[
		\p z = (s_{n-2} d_0 x_1, \dots, d_0 h, \dots, s_{n-2} d_0 x_{n-1}, l, \underline{l'}).
	\]
	Now, if the following $(n+2)$ tuple of $(n-1)$-cycles is compatible, the proof is finished, because then
	$b'_{l'}$ is filled due to lemma \ref{lem:matrix} and all most of the other columns are boundaries by definition:
	\[
		(\p z, \p (s_{n-1}x_1), \dots, h, \dots, \p (s_{n-1}x_{n-1}), b, b'_{l'}) \in (K_{n-1}^{n+1})^{n+2}
	\]
	One can check compatibility using the corresponding matrix:
	\[ 
		\left(
		\begin{matrix}
		   s_{n-2}d_0x_0 \\ \vdots 				\\   d_0h  	\\   \vdots  	\\  s_{n-2}d_0x_{n-1} 		\\ 	l	 	\\ l'
		\end{matrix}\quad
		\begin{matrix}
		   d_0s_{n-1}x_1 \\ d_1s_{n-1}x_1 		\\   \vdots \\   \vdots		\\  d_{n-2}s_{n-1}x_1 		\\ 	x_1 	\\ x_1
		\end{matrix}\;
		\dots\,
	\color{blue}
		\begin{matrix}
			    d_0 h	 \\ s_{n-2}d_1x_k		\\   \vdots \\   \vdots 	\\ s_{n-2}d_{n-2}x_k		\\  x_k    	\\ x_k'
		\end{matrix}\,
		\dots\;
	\color{black}
		\begin{matrix}
		   d_0s_{n-1}x_{n-1}\\ d_1s_{n-1}x_{n-1}\\   \vdots \\   \vdots 	\\  d_{n-2}s_{n-1}x_{n-1}	\\ x_{n-1}	\\ x_{n-1}
		\end{matrix}\quad
		\begin{matrix}
		    l	 		 \\    	\vdots			\\   x_k	\\   \vdots 	\\    \vdots 				\\  x_{n-1} \\ x_n
		\end{matrix}\quad
	\color{blue}
		\begin{matrix}
		    l'	 		 \\    	\vdots			\\   x_k'	\\   \vdots 	\\    \vdots 				\\  x_{n-1} \\ x_n
		\end{matrix}
	\color{black}
		\right).
	\]
	For the other direction we can use the same matrix. But now $l'$ is already given by $b'_{l'}$ 
	and we have to define $d_0h$ by filling 
	\[
		\p z = (s_{n-2} d_0 x_1, \dots, \underline{d_0 h}, \dots, s_{n-2} d_0 x_{n-1}, l, l') \text{ in } L_\pt.
	\]
\end{proof}

	As in corollary \ref{cor:homequi} we now easily get the following corollary. To ensure reflexivity one should define $x\sim x$
	in the case that $d_0x \not\in L_\pt$, which is not clear from definition \ref{def:relhom}.

\begin{corollary}\label{cor:relhomequi}
	Homotopy relative to a subcomplex is an equivalence relation.
\end{corollary}

\begin{remark}\label{rem:homotopies}
	One can define a homotopy between two simplical maps of simplicial sets $f,g:K_\pt \to L_\pt$ 
	to be a simplicial map $h:K_\pt \times \varDelta^1 \to L_\pt$ such that 
	$h_{| K_\pt \times \{0\}} = f$ and $h_{| K_\pt \times \{1\}} = g$. 
	This is an equivalence relation, if $L_\pt$ is a Kan complex. 
	
	In the case $K_\pt = \varDelta^n$ the maps $f$ and $g$ represent the $n$-simplices $\image f, \image g \in K_n$.
	They have the same boundary, if the maps agree on $\partial \varDelta^n$: $f_{|\partial \varDelta^n} = g_{|\partial \varDelta^n}$.
	Then we have $\image f \sim \image g$ in the sense of our definition, if and only if there is a 
	homotopy $h:f\sim g$ that is constant on $\partial \varDelta^n \times \varDelta^1$.
	
	We will not work out these things, as they require additional work. Yet, it is a very interesting
	enhancement to several things we are going to consider, so I recommend to have a look at May's textbook \cite{May},
	where this is elaborated in detail.
\end{remark}

\section{Homotopy Groups}\label{sec:homgrp}

\begin{definition}
	For all $n\in \N_0$ the $n$th homotopy group of a pointed Kan complex $(K_\pt,\s)$ is 
	the set of $n$-simplices with boundary in the basepoint, up to homotopy: 
	\[
		\pi_n(K_\pt,\s) := \lbrace x \in K_n \;|\; \p x = (\s,\dots,\s) \rbrace\big/\sim
	\]
	For $n\ge 1$ the composition of two equivalence classes $[x],[y]\in \pi_n(K_\pt,\s)$ is defined by 
	$[xy]$ where $xy$ completes 
	\[
		(\s,\dots,\s,y,\underline{xy},x).
	\]
	Similarly for a pointed Kan pair $(K_\pt,L_\pt,\s)$ and $n\in \N_{>0}$ the $n$th relative 
	homotopy group is the set
	\[
		\pi_n(K_\pt,L_\pt,\s) := \lbrace x \in K_n \;|\; \p x = (l,\s,\dots,\s), l\in L_\pt \rbrace\big/\sim_{\op{rel}L},
	\]
	and for $n\ge 2$ $[x][y] = [xy]$ is defined by completing the next horn for some $h\in L_{n-1}$:
	\[
		(h,\s,\dots,\s,y,\underline{xy},x).
	\]
	For a map $f:(K_\pt,\s) \to (L_\pt,\s)$ we set $f_* = \pi_n(f): \pi_n(K_\pt,\s)\to \pi_n(L_\pt,\s)$ by $f_*[x]=[f(x)]$.
\end{definition}

\begin{remarque}\label{rem:tophomgrp}
	For a topological space $X$, its topological homotopy groups can be defined 
	as the combinatorial ones of its singular complex:
	\[
		\pi_n^\Top(X,x_0) := \pi_n(S_\pt(X),S_\pt(x_0)), \quad \pi_n^\Top(X,Y,x_0) := \pi_n(S_\pt(X),S_\pt(Y),S_\pt(x_0)).
	\] 
	This identification can be used to motivate the definition of the multiplication in $\pi_n$ by the following picture:
	
	\begfig{0.0cm}
	\center
	\includegraphics[width=0.3\textwidth]{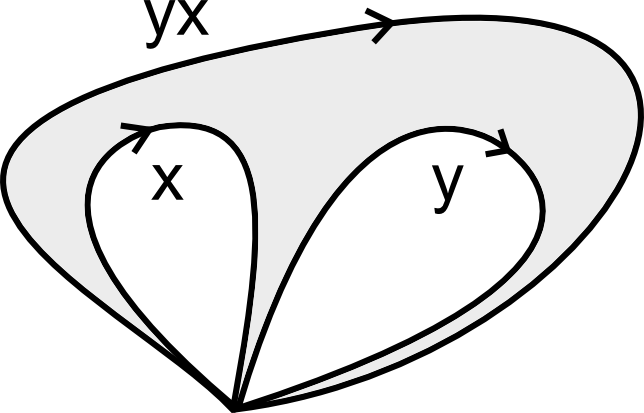}
	\Figure3{The concatenation of two loops $x$ and $y$ is homotopic to their product $xy$.}
	\endfig
	
	The homotopy groups are invariant under the homotopies from remark \ref{rem:homotopies}:
	I.e. if $f\sim g:(K_\pt,\s) \to (L_\pt,\s)$ then $f_*=g_*$. 
	This is for example proven by May in \cite{May}.
	Using the approach of Goerss and Jardine in \cite{GJ}, this statement follows almost immediately.
	Combined with theorem \ref{thm:|S|} this shows that the homotopy groups are also compatible with the geometric realization; 
	i.e. there is a natural isomorphism:
	\[
		\pi_n(K_\pt,L_\pt,\s) \cong \pi_n(S_\pt|K_\pt|,\;S_\pt|L_\pt|,\;S_\pt|\s|) = \pi_n^\Top (|K_\pt|,|L_\pt|,|\s|).
	\]
\end{remarque}

\begin{remarque}\label{rem:motiv}
	Let $ (h,\s,\dots,y,\underline{xy},x) $ be filled as in the definition, then $\p h$ has to be
	\[ 
		(\s, \dots, \s, d_0 y, d_0(xy), d_0 y). 
	\] 
	Thus we get $[d_0x][d_0x]=[d_0(xy)]$
	and so we have found that $d_0: \pi_n(K_\pt,L_\pt,\s) \to \pi_{n-1}(L_\pt)$ is a group homomorphism.
	It is well-defined due to remark \ref{rem:relcase},
	but we do not have yet that the homotopy groups are really well defined groups. 
\end{remarque}

\begin{theoreme}\label{thm:homgrp}
	The (relative) homotopy groups are well defined sets for $n\ge 0\;(1)$ and groups for $n\ge 1\;(2)$, 
	these groups are abelian for $n\ge 2\;(3)$. The assignments $(K_\pt,\s) \mapsto \pi_n(K_\pt,\s)$
	and $(K_\pt,L_\pt,\s) \mapsto \pi_n(K_\pt,L_\pt,\s)$ are functorial.
\end{theoreme}
\begin{proof}
	Due to corollary \ref{cor:homequi} and \ref{cor:relhomequi} we know that homotopy is an equivalence
	relation, thus the homotopy groups are well defined as sets.
	
	For the following proof, let $x,y,z,x',y',z'\in K_n$ be representatives of $\pi_n(K_\pt,\s)$,
	so $\p x = \p y = \dots = (\s,\dots,\s)$.
	
	We first restrict our attention to the case of the absolute homotopy groups.
	\\
	\\
	\emph{Claim 1: The composition is well-defined.}
	
	If we fix two representatives $x,y$ this is just the statement of the Key-lemma \ref{lem:homotopy}.
	But if $x'$ is another representative of $[x]$ then the same lemma tells us that
	\[
		(\s,\dots,\s,y,xy,x') \text{ is filled, because } (\s,\dots,\s,y,xy,x) \text{ is filled.}
	\]
	That shows $[x][y']=[xy]$. If $y'\in[y]$, we easily get $[x'][y'] = [xy]$.
	\\
	\\
	\emph{Claim 2: The group laws hold.}
	
	A neutral element is given by $e=[\s]$ because $\p s_n x = (\s,\dots,\s,x,x)$ implies $[x][\s]=[x]$ and
	$s_{n-1}x$ shows $[\s][x]=[x]$.
	For a given $x$ we find a left-inverse element by filling the horn $(\s,\dots,\s,x,\s,\underline{y})$,
	since this construction yields $[y][x]=[\s]=e$. 
	
	Let $[x][y]=[xy],\;[xy][z]=[(xy)z]$ and $[y][z]=[yz]$. Then, in the following matrix, the second to last column is filled,
	because all the others are:
	\[
	\left(
		\begin{matrix}
			\ddots 	\\		\dots 	\\		\dots 	\\ 		\dots	\\   	\dots	
		\end{matrix}\quad
		\begin{matrix}
			\vdots 	\\		\s 		\\		\s 		\\ 		\s		\\   	\s	
		\end{matrix}\quad
		\begin{matrix}
		\vdots 	\\		\s 		\\		z 		\\ 		yz		\\   	y	
		\end{matrix}\quad
		\begin{matrix}
			\vdots 	\\		\s 		\\		z		\\ 		(xy)z	\\ 		xy
		\end{matrix}\quad
\color{blue}
		\begin{matrix}
			\vdots 	\\		\s 		\\		yz		\\ 		(xy)z	\\ 		x
		\end{matrix}
\color{black}\quad
		\begin{matrix}
			\vdots 	\\		\s 		\\		y		\\ 		xy		\\ 		x	
		\end{matrix}
	\right).
	\]
	This shows $[x][yz]=[(xy)z]$. For the next proofs we will omit lines and collums only containing stars, 
	as they do not provide additional information anyway.
	\\
	\\
	\emph{Claim 3: These groups are abelian for $n\ge 2$.}
	
	First we show that for arbitrary $x,y$ the cycles $(x,x,y,y), (x,\s,\s,x)$ and $(x,y,y,x)$ 
	are filled, by using that $\p(s_ix)=(\dots,\s,x,x,\s,\dots)$ is a boundary:
	\[
	\left(
		\begin{matrix}
			x 		\\		x 		\\ 		\s		\\   	\s	
		\end{matrix}\quad
		\begin{matrix}
			x 		\\		x 		\\ 		\s		\\   	\s	
		\end{matrix}\quad
	\color{blue}
		\begin{matrix}
			x	 	\\		x		\\ 		y		\\ 		y
		\end{matrix}
	\color{black}\quad
		\begin{matrix}
			\s 		\\		\s		\\ 		y		\\ 		y	
		\end{matrix}\quad
		\begin{matrix}
			\s 		\\		\s		\\ 		y		\\ 		y	
		\end{matrix}
	\right)
	\quad
	\left(
		\begin{matrix}
			\s		\\		x 		\\ 		x		\\   	\s	
		\end{matrix}\quad
		\begin{matrix}
			\s 		\\		\s		\\ 		x		\\   	x	
		\end{matrix}\quad
	\color{blue}
		\begin{matrix}
			x	 	\\		\s		\\ 		\s			\\ 		x
		\end{matrix}
	\color{black}\quad
		\begin{matrix}
			x 		\\		x		\\ 		\s		\\ 		\s	
		\end{matrix}\quad
		\begin{matrix}
			\s 		\\		x		\\ 		x		\\ 		\s	
		\end{matrix}
	\right)
	\quad
	\left(
	\color{blue}
		\begin{matrix}
			x		\\		y 		\\ 		y		\\   	x	
		\end{matrix}
	\color{black}\quad
		\begin{matrix}
			x		\\		\s		\\ 		\s		\\   	x	
		\end{matrix}\quad
		\begin{matrix}
			y	 	\\		\s		\\ 		\s		\\ 		y
		\end{matrix}\quad
		\begin{matrix}
			y 		\\		\s		\\ 		\s		\\ 		y	
		\end{matrix}\quad
		\begin{matrix}
			x 		\\		x		\\ 		y		\\ 		y	
		\end{matrix}
	\right).	
	\]
	Now we can use this in the following matrix to show that $[x][y]=[y][x]$:
\[
	\left(
		\begin{matrix}
			x 		\\		\s	 	\\ 		\s		\\   	x	
		\end{matrix}\quad
		\begin{matrix}
			x 		\\		x 		\\ 		y		\\   	y	
		\end{matrix}\quad
		\begin{matrix}
			\s	 	\\		x		\\ 		yx		\\ 		y
		\end{matrix}\quad
	\color{blue}
		\begin{matrix}
			\s		\\		y		\\ 		yx		\\ 		x	
		\end{matrix}
	\color{black}\quad
		\begin{matrix}
			x 		\\		y		\\ 		y		\\ 		x	
		\end{matrix}
	\right).	
	\]
	\\
	\emph{Claim 4: This construction is functorial.}
	 
	To see that $f_*\:\pi_n(K_\pt,\s)\to\pi_n(K_\pt',\s)$ is well-defined by $f_*[x] = [f(x)]$,
	take two representatives with $x\sim y\in K_n$. 
	Given a homotopy $ K_{n+1}\ni h\:x\sim y$, we regard $f(h)$: 
	\[
		\p f(h) = f^{\times n}(\p h) = (f(s_{n-1}d_0x), \dots, f(x), f(y)) = (s_{n-1}d_0f(x), \dots, f(x), f(y))
	\]
	because $f$ commutes with $d_i$ and $s_i$. So $f(h)\:f(x)\sim f(y)$ and therefore $f_*([x]) = f_*([y])$.
	
	To check that $f_*$ is a group homomorphism, let $(\s,\dots,\s,y,xy,x)$ be a boundary 'witnessing' 
	the composition $[x][y]=[xy]$. If it is filled by some $h\in K_{n+1}$ we get 
	\[
		\p (f(h)) = (\s,\dots,\s,f(y),f(xy),f(x)).
	\]
	Thus we have $[f(x)][f(y)]=[f(xy)]$.	
	\\
	\\
	\emph{Claim 5: The relative homotopy groups also satisfy the required properties.}
	
	The proofs in the relative case are obtained by adding a new column and row to the preceding matrices.
	To give an example we do so for the associativity: Let $x,y,z$ be representatives for $\pi_n(K_\pt,L_\pt)$
	and let the following cycles be filled:
	\[
		(l_0,\s,\dots,\s,z,yz,y),\;  (l_1,\s,\dots,\s,z,(xy)z,xy),\;  (l_3,\s,\dots,\s,y,xy,x).
	\]
	This implies $[y][z]=[yz]$ and so on. Then we define $l_2\in L_n$ to fill 
	\[
		(\s,\dots,\s,l_0,l_1,\underline{l_2},l_3),
	\]
	which is a horn because all the $l_i$ have trivial boundary. Now the following matrix is compatible
	and every column but the second to last is certainly filled.
	\[
	\left(
		\begin{matrix}
			\ddots	\\		l_0	 	\\		l_1	 	\\ 		l_2		\\   	l_3	
		\end{matrix}\quad
		\begin{matrix}
			\ddots 	\\		\ddots	\\		\dots	\\ 		\dots	\\   	\dots	
		\end{matrix}\quad
		\begin{matrix}
			l_0 	\\		\vdots	\\		z 		\\ 		yz		\\   	y	
		\end{matrix}\quad
		\begin{matrix}
			l_1 	\\		\vdots	\\		z		\\ 		(xy)z	\\ 		xy
		\end{matrix}\quad
	\color{blue}
		\begin{matrix}
			l_2 	\\		\vdots	\\		yz		\\ 		(xy)z	\\ 		x
		\end{matrix}
	\color{black} \quad
		\begin{matrix}
			l_3 	\\		\vdots	\\		y		\\ 		xy		\\ 		x	
		\end{matrix}
	\right).
	\]	
	Thus it shows $[x][yz]=[xy][z]$. The other proofs are copied similarly from the absolute case.
	One should remark that in the relative case, always one dimension more is needed to carry 
	the simplices from the subcomplex $L_\pt$. This is the case for the definition as a set, 
	the group structure and the commutativity.
\end{proof}

\begin{remarque}\label{rem:loopspace}
	One easily shows that there is a bijection $\pi_n(K_\pt,\s,\s) \cong \pi_n(K_\pt,\s)$, 
	which is a group isomorphism for $n\ge 2$.
	Also, the homotopy groups of a (co-)product can easily be calculated, the reader might try this as an exercise.
	
	Following the general rule from remark \ref{rem:noquo} it should be easy to define the loop-space of a pointed
	Kan complex $(K_\pt,\s)$. Indeed, one can define a path-space by
	\[
		P(K,\s)_n := \{ x \in K_{n+1} \;|\; d_0 \dots d_{n} x = \s \}
	\]
	with the same operators $d_i^P := d_{i}$ and $s_i^P := s_{i}$. Then $d_{n+1}$, which is not used in $P$, induces a simplicial map: 
	$d_{n+1}: P(K,\s)_\pt \to K_\pt$. The preimage of the base-point is the loop-space $\gO(K_\pt,\s) := d_{n+1}^{-1}(\s)$.
	One can then check that $P$ always has trivial homotopy groups
	and the homotopy groups of $\gO$ are those of $K_\pt$ shifted by one dimension. 
\end{remarque}

\section{Applications: The exact sequence of a pair and minimal complexes}\label{sec:exact}

\begin{theoreme}
	The following natural sequence of a pointed pair $(K_\pt,L_\pt,\s)$ is exact:
	\[
		\dots 	\xrightarrow{p_*}	\pi_{n+1}(K_\pt, L_\pt,\s)  
				\xrightarrow{\p_*} 	\pi_n(L_\pt,\s) 
				\xrightarrow{i_*} 	\pi_n(K_\pt,\s) 
				\xrightarrow{p_*} 	\pi_n(K_\pt, L_\pt,\s)  
				\xrightarrow{\p_*} 	\dots
	\]
	\[
		\dots 	\xrightarrow{p_*}	\pi_1(K_\pt, L_\pt,\s)  
				\xrightarrow{\p_*} 	\pi_0(L_\pt,\s) 
				\xrightarrow{i_*} 	\pi_0(K_\pt,\s) 
				\rightarrow			\pi_0(K_\pt,\s)\big/\pi_0(L_\pt,\s).
	\]
	Here $i\:L_\pt \to K_\pt$ is the inclusion, $p_*\:\pi_n(K_\pt,\s)\cong\pi_n(K_\pt, \s,\s)\to\pi_n(K_\pt, L_\pt,\s)$ 
	is due to remark \ref{rem:loopspace}, and $\p_*[x]$ is $[d_0x]$ as in remark \ref{rem:motiv}:
\end{theoreme}
\begin{proof}
	First we show that composing two maps gives the trivial map:\\
	Let $[x] \in \pi_n(K_\pt,\s)$. Then $\p_*p_*[x] = [d_0x] = [\s] \in \pi_{n-1}(L_\pt,\s)$. \\
	If $[x] \in \pi_n(K_\pt,L_\pt,\s)$, we have $\p x = (l,\s,\dots,\s)$ for some $l\in L_\pt$. By the Key-lemma this
	implies $l\sim \s$ in $K_\pt$, 
	so $i_*\p_*[x] = [d_0x] = [\s] \in \pi_{n-1}(K_\pt)$.\\
	For $[l] \in \pi_n(L_\pt,\s)$ we fill $ (\underline{l'},\s,\dots,\s,l)$
	in $L_\pt$ and by definition $\s \sim_{\text{rel}L} l$. Thus $p_*i_*[l]=[l]=[\s]\in\pi_n(K_\pt,L_\pt,\s)$.\\
	\\
	Next, we have to show that the kernels are contained in the images:
	
	Let $i_*[l]=[\s]$ for some $[l]\in\pi_n(L_\pt,\s)$. Because $l\sim\s$ in $K_\pt$, by the Key-lemma we can replace in $(\s,\dots,\s)$ 
	the first star by $l$ and get a filling $h\in K_{n+1}$: $\p h = (l,\s,\dots,\s)$. 
	We can regard $h$ as a representative in $\pi_{n+1}(K_\pt,L_\pt,\s)$, satisfying $\p_*[h]=[d_0h]=[l]$.
	
	Now take $[x]\in\pi_n(K_\pt,\s)$ such that $p_*[x]=[\s]$ or, equivalently, $x\sim_{\text{rel}L}\s$.
	Because $(\s,\dots,\s)$ is filled, the relative Key-lemma finds some $l\in L_n$ for us such that
	$(l,x,\s,\dots,\s)$ is filled in $K_n$. But, since $\p s_0x=(x,x,\s,\dots,\s)$ is filled as well, the 'absolute' Key-lemma 
	tells us $x\sim l$. That is just $i_*[l]=[l]=[x]\in \pi_n(K_\pt,\s)$.
	
	Finally, for $[x]\in\pi_n(K_\pt,L_\pt,\s)$ with $\p_*[x] = [\s]$, take a homotopy $h\:d_0x\sim\s$ in $L_n$.
	Define $y\in K_\pt$ to complete $(h,\s,\dots,\s,x,\underline{y})$, which is a cycle because $d_{n-1}h=d_0x$.
	This implies $\p y=(\s,\dots,\s)$ as well as $x \sim_{\text{rel}L} y$. Therefore $[y]\in\pi_n(K_\pt,\s)$ 
	and $p_*[y]=[y]=[x]\in \pi_n(K_\pt,L_\pt,\s)$.
\end{proof}

As another short application we want to consider minimal and $n$-connected complexes:

\begin{definition}
	A Kan complex $K_\pt$ is called minimal, if for any $x,y \in K_n$ with $\partial x = \partial y$ homotopy implies equality:
	\[
		x \sim y  \Rightarrow x=y.
	\]
	A Kan complex $(K_\pt,\s)$ is called $n$-connected, if we have $\pi_k(K_\pt,\s) = 0$ for all $0\le k \le n$.
\end{definition}

\begin{lemma}\label{lem:minsub}
	Every Kan complex $K_\pt$ has a minimal subcomplex $M_\pt \subset K_\pt$ which is a homotopy retract.
	This means, there is a retraction $r:K_\pt \to M_\pt$ such that $r_{|M} = \id_M$ and there is a homotopy $h:r\sim \id_K$.
	This subcomplex is unique up to isomorphism.
\end{lemma}
\begin{proof}
	This can be found in May's textbook \cite{May} chapter 9, where he also elaborates other facts about minimal complexes.
\end{proof}

\begin{remark}
	Using remark \ref{rem:tophomgrp} the induced map of $r$ yields an isomorphism
	\[
		r_*: \pi_n(K_\pt,\s) \xrightarrow{\sim} \pi_n(M_\pt,\s).
	\]
	The inverse is given by $i_*$, for $i: M_\pt \inj K_\pt$ the inclusion. This works because $(i\circ r)$ and $(r\circ i)$
	both are homotopic to the identity, thus homotopy-invariance tells us that $(i\circ r)_*$ and $(r\circ i)_*$ have to be the identity.
\end{remark}

\begin{lemma}\label{lem:conmin}
	For every $n$-connected and minimal Kan complex $(K_\pt,\s)$ we have that it is trivial up to dimension $n$: 
	$K_k = \{\s\}$ for all $0\le k \le n$.
\end{lemma}
\begin{proof}
	This is easily done by induction on $k$. If $k=0$, for example, 
	then every $x\in K_0$ has to be homotopic to $\s$ because $\pi_0(K_\pt,\s)=\{[\s]\}$.
	On the other hand minimality then implies $x=\s$. The rest is left to the reader.
\end{proof}

\section{Simplicial (Co-)Homology}\label{sec:homology}
We now shortly introduce the simplicial homology and cohomology of a simplicial set. 
References to  a more detailed approach can be found in remark \ref{rem:hom}.

\begin{definition}\label{def:sim-hom}
	We define the simplicial homology $ H_n^\text{sim}(K_\pt,L_\pt;A) $ of a simplicial pair $(K_\pt,L_\pt)\in \sPair$ 
	with coefficients in an abelian group $A$ step by step. It is easily verified in each step that the construction is functorial.
	
	For any $n\in\N$ a simplicial $n$-chain $c$ is a formal $A$-linear combination of $n$-simplices:
	\[
		c = a_1 (x_1) + \dots + a_k (x_k) \quad \text{ here } k \in \N_0, a_i \in A \text{ and } x_i \in K_n.
	\]
	After dividing out simplices $y\in L_n$ of the subcomplex, this forms the group:
	\[
		C_n := C_n^\text{sim}(K_\pt,L_\pt;A) := K_n \otimes A / L_n \otimes A 
		= \{ a_1(x_1)+\dots \;|\; \forall a_i, x_i\}\big/ \left\langle \forall y \in L_n: (y) = 0 \right\rangle.
	\]
	Now, the face maps $d_i$ can be used to construct a differential $\de: C_n \to C_{n-1}$:
	\[
		\de_n (x) := \sum_{i=0}^n (-1)^i (d_i x)\text{ where } x \in K_n \quad \text{ and }\;\; \de_n c = a_1 \de_n(x_1) +\dots
	\]
	We define the set of $n$-boundaries $B_n$ and $n$-cycles $Z_n$ to be these subgroups of $C_n$:
	\[
		Z_n = \ker d_n \text{ and } B_n = \image d_{n+1} .
	\]
	The relation $d_id_j x = d_{j-1}d_i x$ for $i\le j \in [n]$ implies $\de^2 = \de_{n-1} \de_n = 0$,
	in other words $B_n = \image d_{n+1} \subset \ker d_n = Z_n$. So the following definition is possible:
	\[
		H_n^\text{sim} (K_\pt,L_\pt;A) = H_n := Z_n / B_n.
	\]
	This defines a sequence of functors from the category of simplicial pairs to that of abelian groups $H_n: \sPair \to \Ab$.
	We write the equivalence class $[(x)]$ in the quotient $Z_n/B_n$,
	simply as $[x]_H$ or even $[x]$. So $[\de y]_H=0$.
\end{definition}

This is a homology theory in the sense of the \emph{Eilenberg-Steenrod axioms}. 
To formulate them, observe that for each simplicial pair we get canonical maps of pairs:
\[
	(Y,\emptyset) \xrightarrow{i_{XY}} (X,\emptyset) \xrightarrow{\gi_{XY}} (X,Y).
\]

\begin{definition}[Eilenberg-Steenrod Axioms]\label{def:hom-th}
	A homology theory is a sequence of homotopy-invariant, covariant functors $H_n : \sPair \to \Ab \text{ for } n \in \N_{0}$ 
	from the category $\sPair$ to the category of abelian groups $\Ab$, satisfying:
	\begin{itemize}
		\item[Axiom 1:] There is an natural map $\partial_n$ such that the following sequence is exact for any $(X,Y)$:
		\[
			\dots \xrightarrow{H_{n+1}(\gi_{XY})} H_{n+1}(X,Y) \xrightarrow{\partial_{n+1}} H_n(Y,\emptyset) 
			\xrightarrow{H_n(i_{XY})} H_n(X,\emptyset) \xrightarrow{H_{n}(\gi_{XY})} H_n(X,Y) \xrightarrow{\partial_n} \dots
		\]
		\item[Axiom 2:] For any family of simplicial sets $(X_i)_{i\in I}$, we get a natural isomorphism:
		\[
			H_n\left(\coprod_{i\in I} X_i,\emptyset\right) \cong \bigoplus_{i\in I} H_n(X_i,\emptyset) .
		\]
		\item[Axiom 3:] There is a group $A$, called the coefficients of $H_n$, such that for the trivial space $\s$
		\[
			H_n(\s,\emptyset) = \begin{cases} 
					A & \text{ if } n = 0\\
					0 & \text{ if } n \ne 0.
			          \end{cases}
		\]
	\end{itemize}
	Similarly a cohomology theory is a sequence of homotopy-invariant, contravariant functors $H^n:\mathcal{CP}air \to \Ab$ 
	satisfying analogues of these axioms.
	In axiom 1 one has to revert all arrows and in axiom 2 the coproducts $\coprod$ and $\bigoplus$ have to be replaced by a product $\prod$.
\end{definition}

\begin{lemma}
	The simplicial homology $H_n^\text{sim} (.,.;A)$ is a homology theory with coefficients $A$ in the sense of Eilenberg and Steenrod.
\end{lemma}
\begin{proof}
	Homotopy-invariance can easily be deduced from the well-known one for the homology of chain complexes.
	To do so, one has to construct a chain-homotopy out of the simplicial homotopy.
	The homomorphism $\partial_n:H_n(X,Y)\to H_{n-1}(Y)$ can be set to
	\[
		\partial_n [x] = \de [x] \in Y_n \otimes A \text{ for } x \in Z_n(X,Y;A)
	\]
	which is well-defined, because we know $ \de[x] = 0 \mod Y_n \Rightarrow \de[x] \in Y_n \otimes A=C_n(Y,\emptyset;A)$.
	
	More details can be found in the references in remark \ref{rem:hom}.
\end{proof}

\begin{remark}\label{rem:hom}
	For any pair of topological spaces $(X,Y)$ the singular homology is usually defined to be
	the simplicial homology of its singular complex:
	\[
		H_n^\text{sing} (X,Y;A) := H_n^\text{sim} (S_\pt(X),S_\pt(Y);A)
	\]
	This is a homology theory on the category of topological spaces. As this is a common way to define homology in the literature,
	one can find a lot of information on this in the most algebraic topology textbooks.
	For example, the approach of Switzer \cite{Sw}[ch. 10] perfectly fits our needs. 
	More details can be found in chapter two of Hatcher's book \cite{Ha}.
	
	Using the equivalence $S_\pt(|K_\pt|) \sim K_\pt$ from \ref{thm:|S|} and the homotopy invariance of $H$,
	we see that the simplicial homology is compatible with the geometric realization as well: 
	\[
		H_n^\text{sim} (K_\pt,L_\pt;A) \cong H_n^\text{sim} (S_\pt(|K_\pt|),S_\pt(|L_\pt|);A) 
		\stackrel{def}{=} H_n^\text{sing} (|K_\pt|,|L_\pt|;A).
	\]
\end{remark}

\begin{definition}
	In a similar manner, the simplicial cohomology $H_\text{sim}^n (K_\pt,L_\pt;A)$ is defined:
	Here an  $n$-cochain $\ga$ is a map, assigning each $n$-simplex of $K_\pt$ an element of $A$, with the equivalence relation that
	two cochains are identified, if they agree on $K_n \setminus L_n$:
	\[
		C^n := \Hom(K_n \otimes \Z,A) \big/ \Hom(L_n \otimes \Z, A) = \Hom_\text{Set}(K_n \setminus L_n, A) .
	\]
	The codifferential $\de^n:C^n\to C^{n+1}$ is $\de^n c = c \circ \de_n$. Cocycles are $Z^n := \ker d^n$, coboundaries are $B^n := \im d^{n-1}$
	and finally the cohomology is:
	\[
		 H_\text{sim}^n (K_\pt,L_\pt;A) = H^n := Z^n / B^n.   
	\]
\end{definition}

\begin{remark}
	All the duals of remark \ref{rem:hom} hold. Especially this is a cohomology theory 
	and it agrees via $S_\pt$ and $|.|$ with the singular cohomology on $\Top$.
	
	If we do not specify $A$ or 'sim', we mean the simplicial (co-)homology with coefficients in $\Z$, 
	if $L_\pt$ is not mentioned, $L_\pt = \emptyset$ is understood:
	\[
		H_n(K_\pt,L_\pt) = H_n^\text{sim} (K_\pt,L_\pt;\Z) \;\;\text{ and } H_n(K_\pt) = H_n^\text{sim} (K_\pt,\emptyset;\Z)
	\]
\end{remark}

\section{The Hurewicz Homomorphism}

\begin{definition}
	The Hurewicz homomorphism $\gp:\pi_n(K_\pt,L_\pt,\s) \to H_n(K_\pt,L_\pt)$ or 
	$\gp:\pi_n(K_\pt,\s) \to H_n(K_\pt,\s_\pt)$ is defined by: $ \gp([x]_\pi) := [x]_H $ .
\end{definition}

\begin{lemma}\label{lem:Hur}
	The Hurewicz map is a well-defined and natural group homomorphism.  
\end{lemma}
\begin{proof}
	Every representative $x$ for an element $[x]_\pi \in\pi_n(K_\pt,L_\pt,\s)$ of the $n$th homotopy group is an $n$-cycle in homology:
	$\p x = (l,\s,\dots,\s)$ implies $\de [x] = [l] - \sum \pm [\s] = 0$ since $[l]=[\s]=0$ because $l,\s \in L_n$.
	
	If $h:x\sim_{\op{rel}L} y$ is a homotopy, then $\p h = (l,\s,\dots,\s,x,y)$ implies 
	\[ 
		\de [h] = [l] -  [\s]+\dots \pm[\s] \mp [x] \pm [y] = \mp ([x] - [y])
	\] 
	and since any boundary is $0$ in homology, we get $\pm([x]_H-[y]_H)=0 \Rightarrow [x]_H=[y]_H$. So $\gp$ is well-defined.
	Similarly it is a group homomorphism, because $\p h = (\dots,\s,y,xy,x)$ leads to $[x]_H+[y]_H=[xy]_H$.
	Naturality is clear.
\end{proof}

\begin{theoreme}\label{thm:hur}
	If $K_\pt$ is trivial in the dimensions lower than $n$: $K_{n-1} = \{\s\}$, 
	then the Hurewicz homomorphism is an isomorphism between the $n$th homotopy group and the $n$th homology group for $n\ge 2$.
	In the case $n=1$ one has to take the abelization $\Ab(\pi_1(K_\pt,\s))$ of the first homotopy group.
\end{theoreme}

\begin{corollary}[Hurewicz Theorem]
	(1): For any $(n-1)$-connected Kan complex $(K_\pt,\s)$ the Hurewicz homomorphism is an isomorphism in dimension $n$.
	
	(2): The usual Hurewicz theorem also holds:
	For any $(n-1)$-connected topological space $(X,\s)$ the Hurewicz homomorphism is an isomorphism in dimension $n$:
	\[
		\gp: \Ab(\pi_n(K_\pt,\s)) \xrightarrow{\sim} H_n(K_\pt,\s_\pt) 
		\; \text{ and } \;\gp: \Ab(\pi_n(X,\s)) \xrightarrow{\sim} H_n(X,\s_\pt).
	\]
\end{corollary}
\begin{proof}
	(2) immediately follows from (1), because the homotopy and homology groups of a topological space
	are by definition those of its singular complex.
	
	(1) can be reduced to theorem \ref{thm:hur} by choosing a minimal subcomplex $M_\pt \subset K_\pt$
	via lemma \ref{lem:minsub}. Then $M_\pt$ is $(n-1)$-connected, due to the homotopy-invariance of the homotopy groups.
	As we have seen in lemma \ref{lem:conmin} this implies $M_{n-1} = \{\s\}$, since $M_\pt$ is minimal.
	So we can apply the theorem to $M_\pt$ and translate the result back to $K_\pt$ because the homology groups are 
	homotopy-invariant as well.
	
	Here we make essential use of homotopies, homotopy-invariance and minimal complexes,
	although we have not worked out these concepts. Yet, as mentioned in remark \ref{rem:homotopies}, 
	they can be found in \cite{May} and are not too hard to understand.
\end{proof}

\begin{proof}[of theorem \ref{thm:hur}]
	From lemma \ref{lem:Hur} we know that $\gp$ is a group homomorphism. Assume that $\pi_n(K_\pt,\s)$ is abelian, 
	the case $n=1$ is then a slight generalization.\\
	
	\textit{Surjective:} The $n$th homology of $K_\pt$ is generated by $[x]_H$ for $x\in K_n$. The condition $K_{n-1} = \{\s\}$
	forces $\p x$ to be $(\s,\dots,\s)$. So $x$ can also be regarded as a representative in $\pi_n(K_\pt,\s)$ and of course 
	$\gp([x]_\pi) = [x]_H$.\\
	
	\textit{Injective:} It is enough to show that '$\de y = 0$' holds in the homotopy group for $y\in K_{n+1}$. 
	This means that we have to show that 
	\[
		[d_0y]_\pi - [d_1y]_\pi \pm \dots = [\s]_\pi = 0.
	\]
	For $n=1$ this is just the definition of addition, because $\p y=(a,b,c)$ witnesses $[a]_\pi + [c]_\pi = [b]_\pi$.
	So we need the following generalization:
	
	\textit{Claim:} $(\dots,\s,x_k,\dots,x_{n+1})$ is filled if and only if $[x_k]_\pi - [x_{k+1}]_\pi \dots = 0$.
	The proof is done by decreasing induction on $k$. The base $k=n-1$ works just as $n=1$ above.
	
	To get the inductive step we regard the following matrix:
	\[
	\left(
		\begin{matrix}
			\ddots 	\\	\dots 	\\ 	\dots 	\\   	\dots 	\\	\dots 	\\	\dots 	\\  \
		\end{matrix}\quad
		\begin{matrix}
			\vdots 	\\	\s 	\\ 	\s 	\\ 	a	\\ 	a	\\	\s 	\\  \vdots 
		\end{matrix}\quad
		\begin{matrix}
			\vdots 	\\	\s 	\\ 	a'b	\\ 	b	\\ 	a	\\	\s 	\\  \vdots 
		\end{matrix}\quad
	\color{blue}
		\begin{matrix}
			\vdots 	\\	\s 	\\ 	a'b	\\ 	x_{k+2} \\ 	x_{k+3}	\\	x_{k+4}	\\ \vdots
		\end{matrix}\quad
		\begin{matrix}
			\vdots 	\\	a	\\ 	b	\\ 	x_{k+2} \\ 	x_{k+3}	\\	x_{k+4}	\\ \vdots
		\end{matrix}
	\color{black}\quad
		\begin{matrix}
			\vdots 	\\	a	\\ 	a	\\ 	x_{k+3} \\ 	x_{k+3}	\\	\s 	\\  \vdots
		\end{matrix}\quad
		\begin{matrix}
			\vdots 	\\	\s 	\\ 	\s 	\\ 	x_{k+4} \\ 	x_{k+4}	\\	\s 	\\  \vdots
		\end{matrix}\quad
		\begin{matrix}
			\ 	\\	\dots 	\\ 	\dots 	\\ 	\dots 	\\ 	\dots	\\	\dots 	\\  \ddots
		\end{matrix}
	\right) 
	\]	
	Here $a'$ is a representative of $-[a]$. By induction $(\dots, \s, a'b, b, a, \s, \dots)$ is filled.
	We get that $ (\dots,\s,x,x,y,y,\s,\dots)$ is a boundary by generalizing the argument in
	the proof that the higher homotopy groups are abelian, see theorem \ref{thm:homgrp}.  
	
	The matrix now shows that the cycle $(\dots,\s,a,b,x_{k+2},\dots)$ is filled 
	if and only if the cycle $(\dots,\s,\s,a'b,x_{k+2},\dots)$ is filled.
	The latter is, by induction, equivalent to $[a'b]_\pi-[x_{k+2}]_\pi\pm\dots = 0$. 
	This completes the induction, because 
	\[
		[a'b]_\pi-[x_{k+2}]_\pi\pm\dots = -[a]_\pi+[b]_\pi-[x_{k+2}]_\pi\pm\dots\;.
	\]
\end{proof}

\section{Eilenberg-Mac Lane-Spaces}\label{sec:EML}
Here, we will introduce a trivial completion for $n$-skletons. 
It can be used to construct Kan complexes $K_\pt$ with trivial higher homotopy groups.
As a special case, one gets the so-called Eilenberg-Mac Lane spaces, which only have one non-trivial homotopy group.

\begin{definition}\label{def:skeleton}
	A simplicial $n$-skeleton $S_\pt$ is a finite sequence of sets $S_0,\dots,S_n$ with face and degeneracy operators satisfying 
	the rules known from the definition of simplicial sets. 
	Simplicial maps between such $n$-skeletons have to commute with the operators. 
	This defines the category $\sSk_n$.
	
	A completion of such an $n$-skeleton is a simplicial set $K_\pt$ continuing
	the sequence: $K_i = S_i$ for $i\in [n]$, such that the operators are the same. 
	
	An $n$-skeleton is said to have the Kan property,
	if for $k<n$ every $k$-horn is filled and any $n$-horn can be completed to an $n$-cycle.
	Such an $n$-skeleton is called minimal, if it satisfies $x\sim y \Rightarrow x=y$, as for the minimal complexes,
	and if  in addition for $x,y\in S_n$ $n$-simplices $\p x = \p y$ implies $x=y$.
\end{definition}

\begin{remark}
	There is an restriction functor $R_n:\sSet \to \sSk_n$ that simply forgets about all dimensions higher than $n$.
	We can find a adjoint functor $\hat{\cdot}:\sSk_n\to \sSet$, called the completion functor, which extends any $n$-skeleton in a canonical way:
\end{remark}

\begin{lemma}\label{lem:completion}
	There is a canonical completion $\hat{S}_\pt$ for any simplicial $n$-skeleton $S_\pt$, such that if $S_\pt$ satisfies the Kan property, 
	$\hat{S}_\pt$ does so and in addition it has trivial higher homotopy groups: $\pi_k(\hat{S}_\pt,\s) = 0$ for $k\ge n$.
	The completion preserves minimality.
	
	
	This yields a completion functor $\hat{\cdot}:\sSk_n\to \sSet$,
	which is right-adjoint to the restriction functor:
	For any simplicial set $K_\pt$ and skeleton $S_\pt$, there is a natural isomorphism:
	\[
		\Hom_\sSet(K_\pt,\hat{S}_\pt) \cong \Hom_{\sSk_n}(R_nK_\pt,S_\pt).
	\]
\end{lemma}
\begin{proof}
	Let $S_\pt$ be any $n$-skeleton. We then define the completion $K_\pt=\hat{S}_\pt$ by iteratively filling all the $k$-cycles 
	with exactly one simplex of dimension $k+1$ for $k\ge n$:
	\begin{align*}
		K_{k} :=& \{c=(x_0,\dots,x_k) \in K_{k-1}^{k+1} \;|\; \forall i<j\in[k]\: d_ix_j = d_{j-1}x_i \} 
			= \{c \;|\;c\;(k-1)\text{-cycle in } K_\pt\},\\
		&d_i (x_0,\dots,x_k) = x_i \text{ and } s_i c = (s_{i-1} d_0 c, \dots, s_{i-1} d_{i-1} c, c, c, s_{i} d_{i+1} c, \dots, s_i d_k c).
	\end{align*}
	The rules for the simplicial operators are fulfilled. For example by the definition of $s_i$, $\p s_i$ has the right form
	and we have $d_id_j c = d_i x_j = d_{j-1} x_i = d_{j-1} d_i c$.
	
	If $S_\pt$ is Kan, we are interested in the Kan property of the completion. 
	Firstly, observe that for $k\ge n$ by definition every $k$-cycle is filled.
	Thus, to fill an $(k,l)$-horn $(x_0,\dots,\bel,\dots,x_{k+1})$  we just have to complete it to a $k$-cycle. 
	If $k=n$, the Kan property of the skeleton gives such a completion to a cycle.
	For $k>n$, we simply fill 
	\[
		(d_{l-1} x_0, \dots, d_{l-1} x_{l-1},d_l x_{l+1}, \dots, d_l x_{k+1}).
	\]
	by some $x_l$. Then $(x_0,\dots,x_l,\dots,x_{k+1})$ is a $k$-cycle and thus filled.
	
	It is clear that minimality is preserved in the dimensions lower than $n$. For the $n$th dimension we get it,
	because here by the skeleton-minimality no two different simplices have the same boundary. 
	In higher degrees the same rule holds due to the construction.
	
	The construction is clearly functorial. 
	To check the adjointness of the functors we have to uniquely extend any $f:R_n L_\pt \to S_\pt$ to a map $L_\pt\to\hat{S}_\pt$. 
	This is done, by defining 
	\[
		f(x):=f^{\times (k+2)}(d_0 x, \dots, d_{k+1} x) = (f(d_0 x),\dots,f(d_{k+1} x)) \text{ for } x\in L_{k+1}, \; k\ge n.
	\]
	This is also mandatory to make $f$ compatible with the face operator, implying uniqueness.
\end{proof}

\begin{definition}
	A Kan complex $(K_\pt,\s)$ is called Eilenberg-Mac Lane-space in dimension $n\ge 1$ for some abelian group $A$, if it
	satisfies 
	\[
		\pi_k(K_\pt,\s) = 
			\begin{cases}
				A 	& \text{ if } k = n	\\
				0	& \text{ if } k \ne n 	.
			\end{cases}
	\]
\end{definition}

\begin{lemma}
	There is a minimal Eilenberg-Mac Lane-space $H(A,n)_\pt$ for any dimension $n\ge 1$ and any abelian group $A$. 
	It is unique up to isomorphism.
\end{lemma}
\begin{proof}
	We do not check uniqueness here explicitly, as it is not essential. A proof can be found in chapter 23 of May \cite{May}.
	Alternatively, the reader could try a proof by induction on the dimension $k$, extending the argument from lemma \ref{lem:conmin}.

	We define a minimal $(n+1)$-skeleton by 
	\begin{align*}
		H(A,n)_k 	= \{\s\} \text{ for } k < n, \quad&
		H(A,n)_n 	= A \text{ and } \\
		H(A,n)_{n+1} 	= \{(a_0,\dots,a_{n+1})\in A^{n+2}& \;|\; a_0 - a_1 + \dots \pm a_{n+1} = 0 \}
	\end{align*}
	and $s_i a = (\dots,e,a,a,e,\dots)$ for $a\in A$ and $e$ the identity in $A$. The $i$th face map in dimension $(n+1)$
	shall pick the $i$th entry of $(a_0,a_1,\dots)$: $d_i(a_0,\dots) = a_i$.
	
	As soon as we have checked that this is a Kan skeleton, we can use lemma \ref{lem:completion}.
	
	Regard an $(n,k)$-horn: $(a_0,a_1,\dots,\bel,\dots,a_{n+1})$. It is completed by $x$,
	if and only if $a_0 - a_1 \pm \dots \pm x \mp \dots = 0$.
	There is exactly one solution $x$, because $A$ is a group.
	
	Now, consider an $(n+1,k)$-horn:
	\[
		\left(
		\begin{matrix}
			a_0^0 	\\ a_1^0 \\   \vdots 	\\ a_{n+1}^0
		\end{matrix}\quad
		\begin{matrix}
			a_0^1	\\ a_1^1 \\   \vdots 	\\ a_{n+1}^1
		\end{matrix}\quad
		\dots
		\bel
		\dots\;\;
		\begin{matrix}
			a_0^{n+1}\\   \vdots \\  a_{n}^{n+1}	\\ a_{n+1}^{n+1}
		\end{matrix}\quad
		\begin{matrix}
			a_0^{n+2}\\   \vdots \\  a_{n}^{n+2}	\\ a_{n+1}^{n+2}
		\end{matrix}
		\right).
	\]
	There is a canonical completion for the $k$th column: $(a_0^k,\dots,a_{n+1}^k)$ using $a_l^k = a_{k-1}^l$ for $k>l$ 
	and $a_l^k = a_k^{l-1}$ for $k\le l$. But it is not yet clear that this satisfies the algebraic equation in the definition
	of $H(A,n)_{n+1}$. 
	If the matrix is completed fulfilling the symmetry, every entry occurs twice with different sign, 
	thus the alternating sum is zero:
	\[
		\sum_{b=0}^{n+1} \sum_{c=0}^{n+2} (-1)^{b+c}a_b^c = 0 
		\quad \Rightarrow \quad
		\sum_{b=0}^{n+1} (-1)^{b}a_b^k = (-1)^{k+1}\sum_{c=0\atop c\ne k}^{n+2} \sum_{b=0}^{n+1} (-1)^{b+c}a_b^c.
	\]
	The $d$th column satisfies the condition of $H(A,n)_{n+1}$, if and only if $\sum_{b=0}^{n+1} (-1)^{b}a_b^d = 0$. 
	Therefore, by the above equation, the new $k$th column has to satisfy the condition, because all the other columns do so.

	Now the skeleton is constructed and can be canonically extended.
	Then, the homotopy groups are trivial in all degrees but $n$. In dimension $n$ we have a canonical map
	\[
		A \to \pi_n(H(A,n)_\pt),\; a \mapsto [a],
	\]
	because all faces of $a$ are trivial. Homotopies $h\in H(A,n)_{n+1}$ have the boundary $\p h = (e,\dots,e,a,b)$ 
	but by definition this implies $a=b$. Similarly, if $h\in H(A,n)_{n+1}$
	witnesses an addition: $\p h = (\dots,e,y,xy,x)$, then $x+y=y+x=xy$.
\end{proof}

\begin{remark}
	As a by-product, the above construction could be used to calculate the (co-)homology of Eilenberg-Mac Lane spaces.
	If $A$ is finite, the rank of the corresponding (co-)chain complex in degree $k$ is the number of $k$-simplices,
	which is bounded by $|A|^{k!/n!}$. 
	For $A$ infinite, the number of simplices has the same cardinality as $A$. 
\end{remark}

\section{Spectral Cohomology}\label{sec:cohom}
\begin{definition}\label{def:muHA}
	We define a simplicial map between the Kan complexes
	\[
		\mu: H(A,n)_\pt \times H(A,n)_\pt \to H(A,n)_\pt,
	\]
	making $H(A,n)_k$ an abelian group for all $k$. $\mu$ is defined in the $n$th dimension by:
	\[
		\mu(a,b) = a\cdot b \text{ for } a,b\in H(A,n)_n=A.
	\]
	In the lower dimensions $\mu$ is trivial. It can be uniquely extended to all higher dimension,
	since $H(A,n)_\pt$ is defined to be a canonical completion. The compatibility with $\p$ implies
	\[
		\mu((a_0,\dots),(b_0,\dots)) = (\mu(a_0,b_0),\dots) \text{ for } (a_0,\dots),(b_0,\dots)\in H(A,n)_k, k>n,
	\]
	because the face operators in the product are $d_i\times d_i$.
\end{definition}

\begin{definition}
	We define the covariant spectral cocycle functor $Z_\text{spec}^n: \sPair \to \Ab$ 
	in dimension $n\in \N_0$ with coefficients in $A\in\Ab$ by
	\[
		Z_\text{spec}^n(K_\pt,L_\pt;A) := \Hom_\sPair \left( (K_\pt,L_\pt), (H(A,n),\s) \right),
	\]
	with the group-structure induced by $\mu$: $\ga + \gb := \mu\circ(\ga \times \gb)$.
\end{definition}

\begin{definition}\label{def:nullhom}
	A simplicial map of pairs $f:(K_\pt,K_\pt')\to (L_\pt,L_\pt')$ is called nullhomotopic, if it can be extended
	to a map of pairs $f':( K_\pt \times \vD_\pt^1,K_\pt' \times \vD_\pt^1) \to (L_\pt,L_\pt')$, with $f_{|K\times [0]}' = f$
	and $f'(K_\pt \times [1]_\pt) = \s$.
	For a simplicial set $K_\pt$, its cone $(CK_\pt,\s)$ is: 
	\[
		CK_\pt := K_\pt \times \vD_\pt^1\big/ K_\pt \times [1]_\pt.
	\]
	Then $f$ is nullhomotopic, if it can be extended to some $f':(CK_\pt,CK_\pt',\s)\to(L_\pt,L_\pt',\s)$.
\end{definition}

\begin{remark}
	The details of these definitions are not hard to check, however the reader is asked to think about them for a moment.
	
	
	After applying $|.|$ to $CK_\pt$ one really gets the cone over $K_\pt$ and 
	therefore the realization of a nullhomotopic map is homotopic to the constant map.
\end{remark}

\begin{lemma}\label{lem:nullhom}
	A simplicial map $f:(K_\pt,K_\pt')\to (L_\pt,L_\pt')$ is nullhomotopic if and only if there is a sequence of maps 
	$g:(K_n,K_n')\to (L_{n+1},L_{n+1}')$ such that $d_ig(x) = g(d_ix)$ for $i\in [n]$ and $d_{n+1}g(x)=f(x)$, where $x\in K_n$.
\end{lemma}
\begin{proof}
	For the only if part take a nullhomotopy $h$ and set $g(x):=h((x,[0,\dots,0,1]))$.
	
	For the if part define $h(x\times [0,\dots,0]) := f(x)$ and 
	\[
		h(x \times [0,\dots,0,1,\dots,1]) = s_{n-1}\dots s_{k} g(d_k \dots d_{n-1} x)
	\]
	where there are $k$ zeros and $n-k\ge 1$ ones in $[0,\dots,0,1,\dots,1]$ and $x\in K_{n-1}$.
	
	Here several details have to be checked by the reader. An alternative is using the path space from remark \ref{rem:loopspace}
	and defining $g$ as a simplicial map $K_\pt \to PL_\pt$ to the path space. 
\end{proof}

\begin{definition}
	The spectral coboundary $B_\text{spec}^n(K_\pt,L_\pt) \subset Z_\text{Sp}^n(K_\pt,L_\pt)$ 
	is defined to be the image of $Z_\text{spec}^n(i) : Z_\text{spec}^n(CK_\pt,CL_\pt)\to Z_\text{spec}^n(K_\pt,L_\pt)$, 
	where $i$ is the inclusion. 
	
	Now, we can define the spectral cohomology functor $H_\text{spec}^n: \sPair \to \Ab$ in dimension $n\in \N_0$ with coefficients in $A\in\Ab$ by
	\[
		H_\text{spec}^n(K_\pt,L_\pt;A) := Z_\text{spec}^n(K_\pt,L_\pt;A) \big/ B_\text{spec}^n(K_\pt,L_\pt;A).
	\]
\end{definition}

\begin{lemma}\label{lem:HAgroup}
	The above construction works as intended and the coboundaries are exactly those cocycles which are nullhomotopic.
\end{lemma}
\begin{proof}
	It is easily seen that $\Hom_\sPair \left( (K_\pt,L_\pt), (H(A,n),\s) \right)$ is an abelian group.
	For a map of pairs $f:(K_\pt,L_\pt) \to (K_\pt',L_\pt')$ we get an induced map:
	\[
		f^*:\Hom_\sPair \left( (K_\pt',L_\pt'), (H(A,n),\s) \right) \to \Hom_\sPair \left( (K_\pt,L_\pt), (H(A,n),\s) \right)
		\text{ by }
		f^*(\ga) = \ga \circ f.
	\]
	This is a group homomorphism because $f$ acts from the right and $\mu$ from the left side:
	\[
		f^*(\ga + \gb)
		= \left(\mu \circ (\ga \times \gb)\right)\circ f
		= \mu \circ \left((\ga \times \gb)\circ f\right)
		= \mu \circ (f^*\ga \times f^*\gb)
		= f^*\ga + f^*\gb
	\]
	
	A coboundary is of the form $i^*\ga = \ga \circ i = \ga_{|(K_\pt,L_\pt)}$ and thus it is the restriction of a map on the cone
	$\ga:  (CK_\pt,CL_\pt)\to(H(A,n),\s)$. Therefore it is clear that a cocycle is a coboundary, if and only if it is nullhomotopic.
	
	The induced map $f^*$ is well-defined on the quotient $H=Z/B$, since we have $f^*B(K_\pt',L_\pt') \subset B(K_\pt,L_\pt)$,
	as for $i^{\prime *}\ga \in B(K_\pt',L_\pt')$ we have $(Cf)^*(i^{\prime *}\ga) = i^{*}(f^*\ga) \in B(K_\pt,L_\pt)$.
\end{proof}

\begin{theorem}\label{thm:sim=spec}
	For any simplicial pair $(K_\pt,L_\pt)$ we have natural isomorphisms:
	\[
		Z_\text{spec}^n(K_\pt,L_\pt) \cong Z_\text{sim}^n(K_\pt,L_\pt) 
		\text{ and }
		B_\text{spec}^n(K_\pt,L_\pt) \cong B_\text{sim}^n(K_\pt,L_\pt).
	\]
	Therefore both homology theories are equivalent:
	\[
		H_\text{spec}^n(K_\pt,L_\pt) \cong H_\text{sim}^n(K_\pt,L_\pt). 
	\]
\end{theorem}
\begin{proof}
	Let $(K_\pt,L_\pt)$ be any simplicial pair.
	There is a natural isomorphism:
	\[
		\gp: C_\text{sim}^n(K_\pt,L_\pt;A) \cong \Hom_\Set(K_n\big/L_n,A) \cong \Hom_{\sSk_n}(R_n(K_\pt,L_\pt),R_n(H(A,n)_\pt,\s)).
	\]
	This is works, since the skeleton $R_n(H(A,n)_\pt)$ is nearly trivial.
	
	\noindent
	\textit{Claim:} $\gp(\ga)$ can be extended to $(K_{n+1},L_{n+1})$ if and only if $\ga$ is a (simplicial) cocycle.
	
	A simplicial extension $f:R_{n+1}(K_\pt,L_\pt) \to R_{n+1}(H(A,n)_\pt,\s)$ of $\gp(\ga)$ 
	has to satisfy:
	\[
		\p f(x)= f^{\times (n+2)}(\p x) = (\ga(d_0 x),\dots, \ga(d_{n+1} x)) \text{ for } x \in K_{n+1}.
	\]
	This determines $f$ completely. The right hand side is only a simplex, if 
	\[
		(\de^*\ga)(x) = \ga(\de x) = \ga(d_0 x)- \ga(d_1 x) \pm \dots \mp \ga(d_{n+1} x) = 0.
	\]
	As we require this for all $x$, $f$ can indeed be defined if and only if  $\de^*\ga = 0$, so exactly if $\ga$ is a cocycle.
	Thus the claim holds and the completion property of $H(A,n)_\pt$ yields:
	\begin{align*}
		Z_\text{sim}^n(K_\pt,L_\pt;A) &\cong \Hom_{\sSk_{n+1}}(R_{n+1}(K_\pt,L_\pt),R_{n+1}(H(A,n)_\pt,\s)) \\
		&\cong  \Hom_{\sPair}((K_\pt,L_\pt),(H(A,n)_\pt,\s) = Z_\text{spec}^n(K_\pt,L_\pt;A).
	\end{align*}
	
	\noindent
	\textit{Claim:} $\gp(\ga)$ is nullhomotopic if and only if $\ga$ is a (simplicial) coboundary.
	
	We use lemma \ref{lem:nullhom} for this. Regard some map $g:K_k \to H(A,n)_{k+1}$, 
	with $d_ig(x) = g(d_ix)$ and $d_{n+1}g(x)=f(x)$. Then one can define 
	\[
		\gb \in C_\text{sim}^{n-1}(K_\pt,L_\pt;A) = \Hom_\Set(K_{n-1}/L_{n-1},A) \text{ by } \gb(x) := g(x) \in A.
	\]
	For any $y\in K_{n}$ we get 
	\begin{align*}
		g(y) &= (d_0g(y),\dots) = (g(d_0 y),\dots,g(d_n y),f(y)) = (\gb(d_0y),\dots,\gb(d_ny),\ga(y)).
	\end{align*}
	This is a simplex if and only if $\gb(d_0y) - \gb(d_1y) \pm \dots (-1)^{n+1}\ga(y) = 0$, 
	or equivalently $(\de^*\gb)(y) = \gb(\de y) = (-1)^n \ga(y)$. Since this holds for all $y$,
	we get that $\ga = \de^* (-1)^n \gb$ is a coboundary if and only if such a $g$ exists in dimension $k=n$. 
	In the other dimensions $g$ can be extended canonically.
	\\
	\\
	This proves the claim, which then implies the equality of the coboundary sets.
\end{proof}

\begin{remark}
	For a topological space $X$ its spectral cohomology is defined by 
	\[
		H_{spec}^k(X,Y;A) := [(X,Y),(H,\s)] = \Hom_\Top((X,Y),(H,\s))/\sim
	\]
	where $H$ is a topological Eilenberg-Mac Lane space with $\pi_k^\Top(H,\s) \cong A$
	and $\sim$ is the topological homotopy of maps. 
	Since the homotopy groups commute with the topological realization, we can choose $H = |H(A,k)_\pt|$.
	Using the strong statements $X\sim |S_\pt(X)|$ and $K_\pt \sim S_\pt(|K_\pt|)$ from theorem \ref{thm:|S|} we have 
	(also using adjointness of $S_\pt$ and $|.|$):
	\begin{align*}
		H_{spec}^k(X;A) &= \Hom_\Top(X,|H(A,k)|)/\!\!\sim\;\;\cong \Hom_\Top(|S_\pt(X)|,|H(A,k)|)/\!\!\sim \\
				&\cong \Hom_\sSet(S_\pt(X),S_\pt(|H(A,k)|))/\!\!\sim\;\; \cong \Hom_\sSet(S_\pt(X),H(A,k))/\!\!\sim \\
				&= H_{spec}^k(S_\pt(X);A).
	\end{align*}
	Since we also now that $H_{sing}^k(X;A) = H_{sim}^k(S_\pt(X);A)$ by definition, 
	we can use theorem \ref{thm:sim=spec} to show that for every topological pair $(X,Y)$ there is a natural isomorphism:
	\[
		H_{spec}^k(X,Y;A) \cong H_{sing}^k(X,Y;A).
	\]
\end{remark}



\end{document}